\documentclass[11pt,a4paper,reqno]{amsart}

\pdfoutput=1
\usepackage{amsmath,amsfonts,amssymb,amsthm,color}  
\usepackage{mathrsfs}
\usepackage[utf8]{inputenc}
\usepackage{graphicx}

\theoremstyle{plain}
\newtheorem{theorem}{Theorem}[section]
\newtheorem{corollary}[theorem]{Corollary}
\newtheorem{lemma}[theorem]{Lemma}

\theoremstyle{definition}
\newtheorem{definition}[theorem]{Definition}

\numberwithin{equation}{section}
\newtheorem*{theorem*}{Theorem} 

\newcommand{\Z}{{\mathbb Z}}
\newcommand{\R}{{\mathbb R}}
\newcommand{\N}{{\mathbb N}}

\DeclareMathOperator{\dist}{dist}

\DeclareMathOperator{\Mod}{mod}

\title[Paraproducts for convex sets]{Paraproducts for bilinear multipliers associated with convex sets}
\author{Olli Saari}
\author{Christoph Thiele}

\address{Mathematical Institute, University of Bonn, Endenicher Allee 60, 53115 Bonn, Germany}
\email{saari@math.uni-bonn.de, thiele@math.uni-bonn.de}

\begin{document}

\begin{abstract}
	We prove bounds in the local $ L^2 $ range for exotic paraproducts
	motivated by  bilinear multipliers associated with convex sets.
	One result assumes an exponential boundary curve. Another one
        assumes a higher order lacunarity condition.
	\end{abstract}

\maketitle


\section{Introduction}
Given a bounded measurable function $ m $ in the plane,
we define the associated bilinear Fourier multiplier operator acting on a pair of one dimensional Schwartz functions $(f,g)$ by
\begin{equation}\label{bilmult}
B_m(f,g) (x) := \iint_{\R^2} m(\xi,\eta) \widehat{f}(\xi) \widehat{g}(\eta) e^{2 \pi i (\xi + \eta )x} \, d \xi d \eta. 
\end{equation}
The constant multiplier $m = 1$ reproduces the pointwise product that maps $(f,g)$ to $fg$, 
and it is a fundamental question to determine whether and for which $ m $ H\"older's inequality,
true for the pointwise product,
extends to $B_m$.  
We focus on bounds
\begin{equation}\label{holder}\|B_m(f,g)\|_{p_3'}\le C_{p_1,p_2,m} \|f\|_{p_1} \|g\|_{p_2}
\end{equation}
in the open local $L^2$ region
\begin{equation}\label{locall2}
2< p_1,p_2,p_3<\infty, \quad \frac{1}{p_1}  + \frac{1}{p_2} + \frac{1}{p_3} = 1,
\end{equation}
where $p_3' = p_3 /( p_3 -1)$ denotes the dual exponent.

A well-understood example is the classical linear Mikhlin multiplier $m$ in two dimensions. 
It leads to the theory of paraproducts, 
for which we refer to the textbooks \cite{MR1456993,MR3052498} and the original references therein. 
Beyond Mikhlin multipliers, 
the bilinear Hilbert transforms are the most prominent examples.
Their essence is captured by choosing $ m $ to be the characteristic function of a half-plane. The bounds \eqref{holder} and \eqref{locall2} for this case were established in \cite{MR1425870,MR1491450}.
They also hold with a constant independent of the slope and location of the associated half-plane \cite{MR2113017,MR1933076},
and the range of exponents can be extended beyond the local $ L^2 $ range as in \cite{MR3453362,MR1689336,MR1619285}.

Our focus here is on multipliers $m$ which are characteristic functions of convex sets rather than just half-planes as in the case of the bilinear Hilbert transforms. 
Certain curved regions to be discussed shortly,
most notably the epigraph of a parabola, 
were considered in \cite{MR2701349}, 
and the bilinear disc multiplier has been studied in \cite{MR2197068}.
The bounds \eqref{holder} and \eqref{locall2} are known for the disc. 
The results in \cite{MR2197068} also imply the previously known uniform local $L^2$ bounds for the half-plane multipliers
as can be seen by using the invariance of the bounds under dilation and translation of the convex set and approximating half-planes by large discs. 
In \cite{MR3000982}, a lacunary polygon is discussed.

A further list of examples of convex sets
illuminating the increasingly complicated structure can be given as follows.
We say a line has degenerate direction
if it is orthogonal to one of the three vectors $(0,1)$, $(1,0)$ and $(1,1)$. 
A line through a boundary
point of a convex set which does not intersect the interior of the convex set
is a tangent.
We list:

\begin{enumerate}
\item A convex polygon.
\item A convex set of the form $\{(\xi,\eta): 0\le \xi\le 1,\ \gamma(\xi) \le \eta \}$ for some
convex function $\gamma:[0,1]\to [0,1]$ such that all tangent lines
at points $(\xi, \gamma(\xi))$ with $0<\xi< 1$ have slope between $a$ and $2a$ for some $0<a<1$. 
\item A bounded convex set with  $C^2$ boundary curve such that the curvature of the boundary is nonzero at every boundary point with a degenerate tangent line.
\item  A bounded convex set $C$ such that every boundary point with a degenerate tangent line has tangent lines for every direction in a neighborhood of this degenerate direction.
\item The convex set $\{(\xi,\eta): \xi\le 0,\ 2^\xi \le \eta \}$.
\item The convex set $\{(\xi,\eta): 0\le \xi \le 1,\  \gamma (\xi)\le \eta \}$, where $\gamma$ is a monotone increasing convex function mapping $[0,1]$ to $[0,2^{-8}]$ such that the set $\{\gamma(2^{-j}), j\in \N\}$ is multi-lacunary.
See Definition \ref{def:fourierlac} below for the definition of multi-lacunarity.
\item A general bounded convex set $C$.
\end{enumerate}
The first four examples are known to satisfy bounds \eqref{holder} and \eqref{locall2}. 

The first example is easily reduced to the bilinear Hilbert transforms. 
The reduction works by iteratively cutting the multiplier by a line parallel to a degenerate direction, which results in two new multipliers, 
whose operators can be expressed through the original
bilinear multiplier operator and pre- or post-composition with a linear multiplier such as the Riesz projection. 

The second example is a perturbation of the bilinear Hilbert transforms. 
Thanks to the comparable upper and lower bounds on the slope, 
the methods for the bilinear Hilbert transform can be adapted to this situation \cite{MR2701349,MR2197068}.
The third example includes the case of the parabola and the disc \cite{MR2701349,MR2197068}.
It requires an additional technique to put together infinitely many pieces as in the second example 
as well as an additional bound for a central piece near each degenerate direction.
These additional pieces are what we call exotic paraproducts.
The techniques of \cite{MR2701349,MR2197068} still apply at this level of generality.
Likewise, the fourth example can be handled with similar methods. 
Such multipliers are more general away from the critical directions
but have strongly regulated behavior at the degenerate directions. 

In this paper, we study relaxation of the additional conditions near the degenerate directions.
Beginning with the proof of the first bounds for the bilinear Hilbert transform, 
the common approach to understand bilinear multipliers associated with characteristic functions has been to decompose them into smoother multipliers, paraproducts,
which are singular only at a single boundary point instead of a one-dimensional set.
As this boundary point comes closer to a degenerate tangent direction, 
the relevant paraproducts undergo a deformation.  
At a boundary point with degenerate tangent direction,
one encounters an entirely exotic paraproduct, 
whose structure is closely tied to the behaviour of the boundary in the vicinity of that point. 
In the present paper, in particular due to the local $L^2$ range, 
it seems prudent to consider rougher exotic paraproducts which correspond to characteristic functions of certain staircase sets.

One of the results in the present paper provides bounds for the exotic paraproduct
associated with case five of the previous list.
\begin{theorem}\label{exppara}
 Let $p_1,p_2,p_3$ be as in \eqref{locall2}.
 Define 
 \[m(\xi,\eta):=\sum_{j\in \N} 1_{[-(j+1), -j)}(\xi)1_{[2^{-j}, 1)} (\eta).\]
 Then the operator  \eqref{bilmult}  satisfies the a priori estimate \eqref{holder}.
 \end{theorem}
Here we are also able to complete the passage from a paraproduct estimate to a multiplier bound as in \cite{MR2701349,MR2197068} and reduce the bounds for case five to bounds for case two. 
\begin{corollary}\label{expcor}
  Let $p_1,p_2,p_3$ be as in \eqref{locall2}.
 Let $C$ be the convex set
 \[\{ (\xi,\eta):  \xi \le 0,\ 2^{\xi} \le \eta <1 \} .\]
 Then  $B_{m}$  as in \eqref{bilmult} with $m=1_C$
 satisfies the a priori bounds  \eqref{holder}.
 \end{corollary}
The particular cut-offs at $0$ and $1$ are not important in this theorem. 
One also obtains bounds for similar convex sets with constraints of the type $ a^{\xi}\le \eta$ with $ a \ne 2 $
by applying translation and isotropic dilation to the multiplier. 
Hence the number $ 2 $ has no fundamental importance in the corollary.

Our second result proves bounds for exotic paraproducts related to the sixth case of our list.
\begin{definition}[Multi-lacunarity]
	\label{def:fourierlac}
Let $b$ be a non-negative integer.
We call a finite set $X$ of real numbers $(0,b)$-lacunary, if it consists of a single element.

Let $d$ be a non-negative integer and assume we have already defined 
$(d,b)$-lacunarity. We call a finite set $X$ of real numbers $(d+1,b)$-lacunary,
if it can be partitioned into  two sets $L$ and $O$ such that $L$ is $(d,b)$-lacunary
and for any pair of different points $\xi,\xi'$ in $O$ we have 
\[ \dist (\xi,\xi') \ge 2^{-b} \dist(\xi, L).\] 
 \end{definition}

\begin{theorem}\label{mlt}
Let $b,d\ge 2$ be integers and let $p_1,p_2,p_3$ be as in \eqref{locall2}.
  Assume that we have sequences \[(\eta_j)_{j\in \N},\ (\zeta_j)_{j\in \N},\ (\xi_j)_{j\in \N}\] 
  such that for all $j$
  \begin{align}
  \label{etaspacing}
    2^{ -j}  \le  \eta_j \le \zeta_j< 2^{2-j}, \\
  \label{xispacing}
  0\le \xi_j+2^{6-j}\le \xi_{j-1}.
  \end{align}
  Assume the image $X$ of the sequence $(\xi_j)$ be $(d,b)$-lacunary.
  Then the multiplier operator $B_m$ as in  \eqref{bilmult} with
  \begin{equation}\label{expara}
  m(\xi,\eta)=\sum_j 1_{(0,\xi_j)}(\xi) 1_{(\eta_j, \zeta_j)}(\eta)
  \end{equation}
  satisfies the estimate \eqref{holder} with constant 
  \[C_{p_1,p_2,m}=C_{p_1,p_2,b,d}\]
  that depends on $m$ only through $b$ and $d$.
\end{theorem}  
In this case, however, we are not able to complete the program
as the convex set lacks the regularity relevant for techniques in \cite{MR2701349,MR2197068} to work.
Moreover,
the bounds \eqref{holder} and \eqref{locall2} for the general case seven appear entirely beyond the techniques known to us. 
We are without any bias towards validity or invalidity of these bounds.


Standard paraproducts come with exponentially growing sequences $\xi_j$, $\eta_j$ and $\zeta_j$,
all other configurations with an increasing sequence $\xi_j$ and 
an interlaced pair of increasing sequences $(\eta_j,\zeta_j)$ may be considered exotic.
We emphasize that the point here rests in the growth conditions on these sequences, 
the characteristic functions in place of smoother versions found elsewhere in the literature being a minor modification in the considered range of exponents.

Previous research on various assumptions on the sequences can be found in \cite{MR1945289, MR1979774,MR2413217,MR3337797,MR3763348},
partially in connection with bilinear Hilbert transforms on curves, 
but none of these references appears to go beyond the case $d=1$ of multi-lacunarity.
The multi-lacunarity assumption also appears in similar context elsewhere in harmonic analysis.
For example, it provides a sharp condition on sets of directions 
for which the directional maximal operator is unbounded \cite{MR613033,MR2494456}. 
These are deep facts building on a long history of related results. 
Directional Hilbert transforms in the plane with multi-lacunarity assumptions are considered in \cite{MR3846321}.
Multi-lacunary sets of directions in higher dimensions appear in \cite{MR949003,MR3432267,MR4126303}.

\bigskip

\noindent \textit{Acknowledgement.}\ The authors were funded by the Deutsche Forschungsgemeinschaft (DFG, German Research Foundation) under Germany's Excellence Strategy -- EXC-2047/1 -- 390685813 as well as SFB 1060.
Part of the research was carried out while the authors were visiting the Oberwolfach Research Institute for Mathematics, the workshop Real and Harmonic Analysis.

\section{Proof of Theorem \ref{exppara}}

Throughout this section, we fix $p_1,p_2,p_3$ as in \eqref{locall2}.
We work through a sequence of lemmata,
at the end of which we are ready to prove Theorem \ref{exppara} and Corollary \ref{expcor}.
Define the $V^r$ norm for a function  $h$ on $\Z$ or $\R$ by
\[\|h\|_{V^r}:=\sup_x |h(x)|+\sup_{N,\ x_0<x_1<\dots <x_N} \left(\sum _{n=1}^N |h(x_n)-h(x_{n-1})|^r\right)^{1/r}.\]
For a measurable function $n$, we define the linear multiplier operator
\begin{equation}\label{linmult}
M_n f(x):=\int_\R \widehat{f}(\xi)  n(\xi) e^{2\pi i x\xi} \, dx ,
\end{equation}
and we define the multiplier norms
\[\|n\|_{M^p}=\sup_{\|f\|_{p}=1} \|M_nf\|_p\ .\]
If $n$ is the characteristic function of an interval $I$, we write $M_I$ instead of $M_n$.
We also consider the Hardy--Littlewood maximal operator ${\mathcal M}$.

Our first Lemma \ref{crslemma} concerns a bilinear operator, 
which has almost the form of a bilinear multiplier. 
The occurrence of an external parameter $r$ in the exponential makes the difference. 
The main case is $\alpha=0$,
whereas the side product case $\alpha=1$ will be used to estimate certain error terms.

\begin{lemma}\label{crslemma}
  Let $\rho, \phi$ be Schwartz functions  supported in $[-2,2]$. Let  $r\in \R$ and 
  $\alpha\in \{0,1\}$. Let $\epsilon>0$ be small enough so that 
  \[\left \lvert \frac{1}{2} - \frac{1}{p_1} \right \rvert < \frac{1}{2 + \epsilon}.\] 
  Define the operator
\begin{equation*}
  B(f,g)(x):=
\int_{\R^2} \widehat{f}(\xi)\widehat{g}(\eta)
 e^{2\pi i (\xi x+\eta r)}   \sum_{j\in \N} 2^{-\alpha j} \rho(\xi+j){\phi}(2^{j}\eta )d\xi d\eta
\end{equation*}
and the averages 
\[Ag(r)(j):= \int \widehat{g}(\eta) \phi(2^{j}\eta) e^{2\pi i \eta r} d\eta.\]
Then there is a constant $C_{\rho,\phi,p_1,\epsilon}$  such that for all Schwartz functions $f$ and $g$, we have
\begin{equation}\label{crsbounds}
\|B(f,g)\|_{p_1}\le C_{\rho,\phi,p_1,\epsilon} \|f\|_{p_1} \|Ag(r)\|_{V^{2+\epsilon}}.
\end{equation}
\end{lemma}
\begin{proof}
 We decompose $\rho$ into a Fourier series on an interval of length four
 \[\rho(\xi)=1_{[-2,2]}(\xi) \sum_{k\in \Z/4} \widehat{\rho}_k e^{2 \pi i k \xi}.\] 
 Splitting $B$ correspondingly and using the rapid decay of $\widehat{\rho}_k$,
 we see it suffices to prove bounds analogous to \eqref{crsbounds} on
 \[
\int_{\R^2} \widehat{f}(\xi)\widehat{g}(\eta)
e^{2\pi i (\xi x+\eta r)}   \sum_{j\in \N} 2^{-\alpha j} 1_{[-2,2]}(\xi+j)e^{2\pi i k (\xi+j)}\phi(2^{j}\eta )d\xi d\eta,\]
uniformly in $k\in \Z/4$.

We split the sum over $j$ into four sums depending on the congruence class of $j$ modulo four,
and we notice that the factor $e^{2\pi i k j}$ is constant in $j$ in each of the four sums.
Using further that the translation $x\to x+k$ leaves the $L^{p_1}$ norm invariant,
we see it suffices to prove bounds analogous to \eqref{crsbounds} on 
 \[
\int_{\R^2} \widehat{f}(\xi)\widehat{g}(\eta)
e^{2\pi i (\xi x+\eta r)}   \sum_{j \in \N,\ j\equiv j_0 \Mod 4} 2^{-\alpha j} 1_{[-2,2]}(\xi+j)\phi(2^{j}\eta )d\xi d\eta\]
for a fixed parameter $j_0$. 
 
We identify this expression as a linear multiplier of the form \eqref{linmult} 
applied to $f$. 
The multiplier symbol is
\begin{equation*}
n(\xi)= 
 \sum_{j\in \N, \ j\equiv j_0 \Mod 4} 2^{-\alpha j}  1_ {[-2,2]}(\xi+j) Ag(r)(j),
\end{equation*}
and it thus suffices to show
 \begin{equation}\label{multcont}\|n\|_{M^{p_1}}\le  C_{\rho,\phi,p_1,\epsilon} \|Ag(r)\|_{V^{2+\epsilon}} .
 \end{equation}

 We apply the following well-known control of the multiplier norm by variation norms
 proven by Coifman, Rubio de Francia and Semmes \cite{MR934617}.
\begin{theorem}\label{crs}
  Let $n$ be a measurable function on $\R$, then
   \[\|n\|_{M^p}\le C_{p,r}\|n\|_{V^r}, \]
  provided $1<p<\infty $ and $|1/2-1/p|\le 1/r$.
\end{theorem}  
For $\alpha=0$, inequality \eqref{multcont} follows 
if one identifies $n$ as a step function
constant on intervals of length $4$, 
taking precisely the value $Ag(r)(j)$ in the $j$-th interval,
counted in natural order,  
and taking the value $0$ outside the union of these intervals.

For $\alpha=1$, we observe that for any sequence $a(k)$ tending to zero,
the variation norm of the sequence $b(k)=2^{-k}a(k)$ is bounded by a constant times 
the supremum norm of $a$, which in turn is controlled by the variation norm of $a$.
This fact following from a plain application of the triangle inequality completes the proof of Lemma \ref{crslemma}.
\end{proof}

The next lemma passes to a localized version of actual bilinear multipliers.
The parameter $r$ in the exponent disappears, 
but we introduce a new localization parameter $ s $. 
In what follows, 
we use the translation operator defined by 
\[T_sh(x)=h(x-s).\]

\begin{lemma}\label{leplemma}
Let $\chi, \rho, \phi$ be Schwartz functions supported in $[-2,2]$ and
\begin{equation}\label{lepbound}
  B(f,g)(x):=
 T_s \chi(x) \int_{\R^2} \widehat{f}(\xi)\widehat{g}(\eta)
 e^{2\pi i (\xi+\eta)x}   \sum_{j\in \N} \rho(\xi+j)\phi(2^{j}\eta )d\xi d\eta.
\end{equation}
Given the averages $Ag(r)(j)$ and a parameter $\epsilon$ as in Lemma \ref{crslemma}, 
there is a constant $C_{\chi,\rho,\phi,p_1,\epsilon}$ 
so that for all Schwartz functions $f$ and $g$ we have 
\[\|B(f,g)\|_{p_1}\le C_{\chi,\rho,\phi,p_1,\epsilon} \|f\|_{p_1} {\mathcal M}_2(\|Ag\|_{V^{2+\epsilon}})(s) .\]
\end{lemma}

\begin{proof}
We apply the support assumption on $\chi$ and the
fundamental theorem of calculus as well as the product rule to write \eqref{lepbound} as 
\begin{align}\label{firstftc}
&\int_{s-2}^x  T_s \chi'(r) \int_{\R^2} \widehat{f}(\xi)\widehat{g}(\eta)
 e^{2\pi i (\xi x+\eta r)}   \sum_{j\in \N} \rho(\xi+j)\phi(2^{j}\eta )d\xi d\eta dr \\
 \label{secondftc}
&\ + \int_{s-2}^x  T_s \chi(r) \int_{\R^2} \widehat{f}(\xi)\widehat{g}(\eta)
 2\pi i \eta e^{2\pi i (\xi x+\eta r)}   \sum_{j\in \N} \rho(\xi+j)\phi(2^{j}\eta )d\xi d\eta.\end{align}

We first discuss  the term \eqref{firstftc}.  We estimate the integral in $r$ with the $L^1$ norm 
to obtain the bound
\begin{multline*}
\left\| \left\| T_s \chi'(r) \int_{\R^2} \widehat{f}(\xi)\widehat{g}(\eta)
 e^{2\pi i (\xi x+\eta r)}   \sum_{j\in \N} \rho(\xi+j)\phi(2^{j}\eta )d\xi d\eta \right\|_{L^1(r)} \right\|_{L^{p_1}(x)}
\\
\le  \left\| T_s \chi'(r)  \left\|  \int_{\R^2} \widehat{f}(\xi)\widehat{g}(\eta)
 e^{2\pi i (\xi x+\eta r)}   \sum_{j\in \N} \rho(\xi+j)\phi(2^{j}\eta )d\xi d\eta\right\|_{L^{p_1}(x)}  \right\|_{L^1(r)}
\end{multline*}
for the $L^{p_1}(x)$ norm of \eqref{firstftc}.
We used Minkowski's inequality and the condition $1<p_1$ here.
Next we estimate the inner norm with Lemma \ref{crslemma}. 
Setting $\alpha=0$ and $\epsilon > 0$ small,
we obtain an upper bound by
\begin{multline*} 
  C_{\rho,\phi, p_1,\epsilon}   \left \| T_s \chi'(r)   
  \left \| f  \right\|_{p_1} \left \| Ag(r)) \right\|_{V^{2+\epsilon}}
 \right\|_{L^2(r)} \\
 \le  C_{\rho,\phi,\chi, p_1,\epsilon}  
 \|f\|_{p_1}{\mathcal M}(\|Ag\|_{V^{2+\epsilon}})(s).
\end{multline*}
The last inequality follows by recognizing a smooth average over the support of $T_s\chi'$, 
which is near $s$,  
and dominating it by the Hardy--Littlewood maximal function.
This completes the bound for first term \eqref{firstftc}. 

We rewrite the second term \eqref{secondftc} as
\[ 2\pi i 
 \int_{s-2}^x  T_s \chi(r) \int_{\R^2} \widehat{f}(\xi)\widehat{g}(\eta)
  e^{2\pi i (\xi x+\eta r)}   \sum_{j<0}  2^j\rho(\xi-j)\tilde{\phi}(2^{-j}\eta )d\xi d\eta,\]
where the new Schwartz function $\tilde{\phi}(\eta)=\eta\phi(\eta)$ depends on $\phi$ only.
We can proceed exactly as with the first term \eqref{firstftc},
but now applying Lemma \ref{crslemma} 
with $\alpha=1$ and the Schwartz function $\tilde{\phi}$
instead of $ \alpha = 0 $ and the Schwartz function $ \phi $.
This results in the desired bound 
and hence completes the proof of Lemma \ref{leplemma}.
\end{proof}

Finally, we can pass to a standard bilinear multiplier estimate.
This entails getting rid of the localization present in the previous estimate.
The multiplier below is a smooth model of the one in Theorem \ref{exppara}.

\begin{lemma}\label{expplemma}
Let $\rho$ and $\phi$ be Schwartz functions supported in $[-2,2]$ and
\[B(f,g)(x):=\int_{\R^2} \widehat{f}(\xi)
\widehat{g}(\eta) e^{2\pi i (\xi+\eta)x} \sum_{j\in \N} \rho(\xi+j)\phi(2^{j}\eta)
d\xi d\eta .  \]
There is a constant $C_{\phi, \rho, p_1,p_2}$ such that for any Schwartz functions $f$ and $g$
\[\|B(f,g)\|_{p_3'}\le C_{\phi, \rho, p_1,p_2} \|f\|_{p_1} \|g\|_{p_2}.\]
\end{lemma}

\begin{proof}
It suffices to prove the dual estimate 
\[\int_\R h(x)B(f,g)(x)\, dx \le C_{\phi, \rho, p_1,p_2} \|f\|_{p_1} \|g\|_{p_2}\|h\|_{p_3}\]
for all Schwartz functions $h$.
Fixing a suitable normalized non-negative Schwartz function $\chi$ 
supported on $[-2,2]$,
we write 
\begin{equation}\label{localize}
 \int_\R h(x)B(f,g)(x)\, dx  
=\int_{\R^3} h(x)T_{s} \chi^3(x) B(fT_{s+t} \chi,g)(x)\, dsdtdx,
\end{equation}
which can be done because $\int_\R T_s\chi(x) \, ds$ is a constant function 
as is the integral over $ s $ of $ T_s \chi^{3}(x) $.
Next we seek an estimate for 
\begin{equation}\label{doublelocal}
\|T_{s} \chi  B(fT_{s+t} \chi,g)\|_{p_1}
\end{equation}
for arbitrary real $t$. 

For $|t|\le 4$, 
we use Lemma \ref{leplemma} and an epsilon depending only on $p_1$
to obtain the upper bound 
\[\|T_{s} \chi  B(fT_{s+t} \chi,g)\|_{p_1} \le  C_{\chi,\rho,\phi,p_1} \|f T_{s+t}\chi\|_{p_1} {\mathcal M}(\|Ag\|_{V^{2+\epsilon}})(s) . \]
For $|t|> 4$, we write the Fourier expansion
\begin{multline}
B(fT_{s+t} \chi,g)(x) \label{tlargerfour} \\
=
\int_{\R^3} \widehat{f}(\zeta) \widehat{\chi}(\xi-\zeta )
\widehat{g}(\eta) e^{2\pi i [(\xi+\eta)x-(s+t)(\xi-\zeta)]} \sum_{j\in \N} \rho(\xi+j)\phi(2^{j}\eta)
d\xi d\eta d\zeta . 
\end{multline}
Integrating by parts twice,
we can write the $\xi$ integral in \eqref{tlargerfour} 
as
\[\sum_{\beta=1}^3 \frac{1}{ (x-s-t)^{2}} \int_{\R} \widehat{\chi}_{\beta}(\xi-\zeta ) e^{2\pi i [(\xi+\eta)x-(t+s)(\xi-\zeta)]} \sum_{j<0} \rho_\beta (\xi-j)
d\xi , \]
where the three pairs of Schwartz functions $(\chi_\beta,\rho_\beta)$, $ \beta = 1,2,3 $,
are determined by the product rule.
Set 
\[
\tilde{\chi}(x) = \sum_{\beta = 1} ^{3} |\chi_\beta(x)| + |\chi(x)|.
\]
Note that $|x-s-t|$ is comparable to $|t|$
when $ x \in [s-2,s+2] $,
which is the essential domain of integration in \eqref{localize}. 
Inserting this into \eqref{tlargerfour}
and using Lemma \ref{leplemma},
we obtain for \eqref{doublelocal} the upper bound
\[C_{\chi,\rho,\phi,p_1}  t^{-2}\|f T_{s+t} \tilde{\chi} \|_{p_1} {\mathcal M}( \|Ag\|_{V^{2+\epsilon}})(s) .\]
Combining the two cases of small and large $t$, we obtain for \eqref{doublelocal} the bound
\[ C_{\chi,\rho,\phi,p_1}  (1+t^2)^{-1}\|f T_{s+t} \tilde{\chi} \|_{p_1} {\mathcal M}(\|Ag\|_{V^{2+\epsilon}})(s) .\]

Turning back to \eqref{localize}, 
we apply H\"older's inequality on the integral in $x$ and a trivial bound on a factor $\|T_s\chi\|_{p_2}$  
to estimate \eqref{localize} by 
 \begin{multline*}
C_{\chi,p_2} \int_{\R^2} \|h T_{s} \chi\|_{p_3}
\|T_{s} \chi B(fT_{s+t} \chi,g)\|_{p_1}\, dsdt \\
\le C_{\chi,\rho,\phi,p_1,p_2} \int_{\R^2} \|h T_{s} \chi\|_{p_3}
(1+t^2)^{-1}\|f T_{s+t} \tilde{\chi} \|_{p_1} {\mathcal M}(\|Ag\|_{V^{2+\epsilon}})(s)\, dsdt \\
\le C_{\chi,\rho,\phi,p_1,p_2} \int_{\R} \|h \|_{p_3}
(1+t^2)^{-1}\|f \|_{p_1} \|{\mathcal M}(\|Ag\|_{V^{2+\epsilon}})\|_{p_2}\, dt. 
\end{multline*}
In the last line, we have applied H\"older's inequality and noted
\[\int_\R  \|h T_{s} \chi\|_{p_3}^{p_3}\, ds
  =  \int_{\R^2 } |h(x)T_s\chi(x)|^{p_3} dx ds
  = C_{\chi,p_3}\|h\|_{p_3}^{p_3} . \]
The same computation also applies to the factor with $ L^{p_1} $ norm.
The integral in $t$ is trivial, 
and hence we obtain for \eqref{localize} the upper bound 
\[
C_{\chi,\rho,\phi,p_1,p_2} \|h \|_{p_3} \|f \|_{p_1} \|\mathcal{M}(\|Ag\|_{V^{2+\epsilon}})\|_{p_2} .\]

It remains to observe 
\[\|\mathcal{M}(\|Ag\|_{V^{2+\epsilon}})\|_{p_2}
\le C_{\phi,p_2}  \|g\|_{p_2} . \]
This follows from the Hardy--Littlewood maximal theorem
and a well-known variational bound stated in the following theorem,
whose general formulation we quote from Jones, Seeger and Wright \cite{MR2434308}.
\begin{theorem}
Let $1<p<\infty$ and $2<r<\infty$.
Let $\phi$ be a Schwartz function and define $\phi_t(x)=2^{-t}\phi(2^{-t} x)$.
Define the averaging operators
  \[Ah(x)(t)=\int_\R h(x-y)\phi_t(y)\, dy. \]
  Then we have the bound
  \[\|\|Ah\|_{V^r}\|_{L^p}\le C_{p,r}\|h\|_p.\] 
\end{theorem}
The theorem goes back to L\'epingle in the martingale setting \cite{MR420837}, 
and it was introduced for applications in harmonic analysis by Bourgain \cite{MR812567}.
However, the precise formulation we use is from \cite{MR2434308}.
Note that the convolution operator appearing in this Theorem 
can be written equivalently as a multiplier corresponding to our definition of $A$.
This completes the proof of Lemma \ref{expplemma}.
\end{proof}

We are now ready to prove Theorem \ref{exppara}.
What remains is to pass from the smooth model in Lemma \ref{expplemma}
to the rough paraproduct defined using characteristic functions of intervals instead of smooth Schwartz bumps.

\begin{proof}[Proof of Theorem \ref{exppara}]
We regroup the multiplier operator of Theorem \ref{exppara} as
\[B_m(f,g)=\sum_{j\in \N} M_{(-\infty,-j-1)} f M_{[2^{-j-1},2^{-j} ) } g .\] 
Pick a smooth function $\rho$ supported in $[-1,1)$ such that 
\[ \sum_{k \in \Z} T_{-k} \rho(k) = \sum_{k\in \Z} \rho(\xi+k)=1. \]
We compare $B_m$ with
\[B_{\tilde{m}}(f,g)=\sum_{j\in \N} \sum_{k>j} M_{T_{-k} \rho} fM_{[2^{-j-1},2^{-j})}g .\]
The difference satisfies
\[(B_{\tilde{m}}-B_m )(f,g)=\sum_{j\in \N} M_{T_{-j-1}\rho} M_{[-j-1,-j]} f M_{[2^{-j-1},2^{-j})} g,\]
and we have the estimate
\begin{multline*}\|(B_{\tilde{m}}- B_{{m}})(f,g)\|_{p_3'} \\
\le \|(\sum_{j\in \N} |M_{T_{-j-1}\rho }M_{[-j-1,-j]   } f |^2)^{1/2}\|_{p_1}
\|(\sum_{j\in \N} |M_{[2^{-j-1},2^{-j})}g|^2)^{1/2}\|_{p_2} .
\end{multline*}
The second factor is bounded by $C_{p_2}\|g\|_{p_2}$ by the following
well-known square function estimate due to Rubio de Francia \cite{MR850681}.
\begin{theorem}\label{rdf}
Let $\mathcal I$ be a collection of pairwise disjoint intervals.
Then, for $2<p<\infty$, there is a constant $C_p$ such that for all Schwartz functions $f$
we have
\[ \|(\sum_{I\in {\mathcal I}}|M_If(x)|^2)^{1/2}\|_p\le C_p\|f\|_p .\]
\end{theorem}

To estimate the first factor, 
we note that the operator $M_{T_{-j-1}\rho}$ is dominated pointwise
by a constant multiple of the Hardy--Littlewood maximal function. 
Using the Fefferman--Stein maximal inequality \cite{MR284802} as well as Theorem \ref{rdf},
we can estimate
\begin{multline*}
  \|(\sum_j (M_{(T_{-j-1}\rho }M_{[-j-1,-j]   } f )^2)^{1/2}\|_{p_1} \\
\le C_{\rho,p_1}\|(\sum_j (\mathcal{M} M_{[-(j+1),-j]   } f )^2)^{1/2}\|_{p_1} \\
\le C_{\rho,p_1}\|(\sum_j ( M_{[-(j+1),-j]   } f )^2)^{1/2}\|_{p_1} \le 
C_{\rho,p_1}\|f\|_{p_1} .
\end{multline*} 

It then remains to estimate $B_{\tilde{m}}$. 
We cut the intervals $ [2^{-j-1},2^{-j}) $ into two equally long halves.
For $ \beta \in \{0,1\} $, define
\begin{align}\label{ee}
m_{\beta}(\xi, \eta)
  &= \sum_{k\in \N}   T_{-k} \rho(\xi)  \sum_{j\in \N, j<k} 
  1_{[2^{-j-1},2^{-j-1} + 2^{-j-2})  +  \beta 2^{-j-2}}(\eta)
\end{align}
so that $ \tilde{m} = m_0 + m_1 \nonumber  $.
We estimate the corresponding two multiplier operators separately.
Set
\begin{equation}
\label{eq:tildeg}
\tilde{g}_\beta =\sum_{j \in \N} M_{[2^{-j-1},2^{-j-1} + 2^{-j-2})  +  \beta 2^{-j-2}}  g, 
\end{equation}
and observe
\[\sum_{j\in \N, j<k} M_{[2^{-j-1},2^{-j-1} + 2^{-j-2})  +  \beta 2^{-j-2}} g= \tilde{g}_\beta-\phi_{\beta,k}*\tilde{g}_\beta,\]
where $\widehat{\phi}_{\beta,k}= \widehat{\phi}_\beta(2^k x)$ and $\widehat{\phi}_{\beta}$ is a smooth function
supported in \[
[-1 ,1] + [-\beta/4 , \beta/4]
\] 
and constant one on \[
[-3/4 , 3/4 ] + [-\beta/4 , \beta/4].
\]
Hence the multiplier operator corresponding to \eqref{ee} becomes 
\begin{equation}\label{eepart}
\sum_{k\in \N } (M_{T_{-k} \rho}f) \tilde{g}_\beta- \sum_{k\in \N} (M_{T_{-k} \rho}f)(\phi_{\beta , k} *\tilde{g}_\beta).
\end{equation}
The second term in \eqref{eepart} is bounded by
\[\|\sum_{k\in \N} (M_{T_{-k} \rho}f)(\phi_{\beta,k} *\tilde{g}_\beta)\|_{p_3'} \le C_{p_1,p_2}  \|f\|_{p_1} \|\tilde{g}_\beta\|_{p_2}
\le C_{p_1,p_2}  \|f\|_{p_1} \|g\|_{p_2} , \]
the first inequality following from Lemma \ref{expplemma} 
and the second inequality being a consequence of the following Theorem \ref{hmt} 
below by Coifman, Rubio de Francia and Semmes \cite{MR934617}.

\begin{theorem}\label{hmt} 
Let  $n$ be a bounded function on the real line.
If
\[
\sup_{j \in \Z}  \left \lVert n 1_{[2^j,2^{j+2}) \cup [-2^j,-2^{j+2})} \right \rVert_{V^1} \le A ,
\]
then for each $1<p<\infty$
there is a constant $C_p$ such that
\[\|n\|_{M^p}\le C_p A. \]
\end{theorem}

This is the sharp version of the H\"ormander--Mikhlin multiplier
theorem, a stronger version of Theorem \ref{crs} above.
The assumption on the total variation is obviously satisfied 
by the multiplier in \eqref{eq:tildeg}. It only jumps a uniformly bounded number of 
times in each of the test intervals,
each of the jumps being of height one.
Hence we have the desired bound for the second term in \eqref{eepart}.

For the first term in \eqref{eepart}, 
we have the bound 
\[\|\sum_{k\in \N} (M_{T_{-k} \rho}f) \tilde{g}_{\beta} \|_{p_3'}\le 
\|\sum_{k\in  \N} (M_{T_{-k} \rho}f) \|_{p_1}
\|\tilde{g} \|_{p_2}\le C_{\rho, p_1}\|f\|_{p_1}\|g\|_{p_2}\]
where we again used Theorem \ref{hmt} for the second factor.
The first factor is estimated by Theorem \ref{crs} as
the corresponding multiplier is locally constant except for in $ [-1,1]$
where it is a smooth function of total variation one.
\end{proof}

\begin{proof}[Proof of Corollary \ref{expcor}]
We write the multiplier in the corollary as
\begin{equation}\label{exolac}
m=\sum_{j\in \N} 1_{[-(j+1), -j)}(\xi)1_{[2^{-j}, 1)} (\eta)+ \sum_{j\in \N} m_{j},
\end{equation}
where $m_j$ is the characteristic function of the set
\begin{equation}
  \label{eq:add_set_132}
  \{ (\xi,\eta):  -(j+1)\le \xi \le -j , 2^{\xi}\le \eta <2^{-j} \} .
\end{equation}
The bilinear multiplier corresponding to the first summand
in \eqref{exolac} is bounded by Theorem \ref{exppara}.

To estimate the second summand, 
we pair the bilinear multiplier with an arbitrary function $h \in L^{p_3}$ 
and write
\begin{multline*}
\int_\R \sum_{j\in \N} B_{m_j}(f,g)(x)h(x)\, dx \\
= \sum_{j\in \N} \int_\R B_{m_j}(M_{[-(j+1),-j)}f,M_{[2^{-(j+1)}, 2^{-j})}g)(x)  \\
 \times M_{[j,(j+1))+[-2^{-j}, -2^{-(j+1)})}h(x)\, dx.
\end{multline*}
Here we have inserted an additional multiplier to restrict the frequency support of $h$.
This is possible as a bilinear multiplier applied to a pair of functions 
with frequency supports in the intervals $I$ and $J$,
results in a function that has frequency support in $-I-J$.

We identify each piece $B_{m_j}$ with a multiplier of the form corresponding to the second case in the list of examples of convex sets in the introduction,
that is, we notice the slope of the curved boundary line of \eqref{eq:add_set_132} is bounded above and below by comparable numbers. 
Each $ B_{m_j} $ is hence individually bounded, 
and we can estimate the display above by
\[  \sum_{j\in \N} \|M_{[-(j+1),-j)}f\|_{p_1}\|M_{[2^{-(j+1)}, 2^{-j})}g\|_{p_2}\|M_{[j,(j+1))+[-2^{-j}, -2^{-(j+1)})}h \|_{p_3}\]
Applying H\"older's inequality, we estimate this by
\begin{multline*} 
  \|(\sum_{j\in \N} |M_{[-(j+1),-j)}f|^{p_1})^{1/p_1}\|_{p_1}  \\
 \times \|(\sum_{j\in \N} |M_{[2^{-(j+1)}, 2^{-j})}g|^{p_2})^{1/p_2}\|_{p_2} \\
 \times  \|(\sum_{j\in \N} |M_{[j,(j+1))+[-2^{-j}, -2^{-(j+1)})}h|^{p_3})^{1/p_3}\|_{p_3} .
\end{multline*}
Because $p_1 > 2$ and the intervals $ [-(j+1), -j)$ are disjoint, 
we can use Theorem \ref{rdf} and estimate the first factor by
\begin{multline*}  
  \|(\sum_{j\in \N} |M_{[-(j+1),-j)}f|^{p_1})^{1/p_1}\|_{p_1}
  \le 
 \|(\sum_{j\in \N} |M_{[-(j+1),-j)}f|^{2})^{1/2}\|_{p_1} \\
\le C_{p_1} \|f\|_{p_1}.
\end{multline*}
Similarly, we can estimate the factors corresponding to $ g $ and $ h $.
The intervals $ [2^{-(j+1)}, 2^{-j}) $ relevant to $ g $ are also disjoint,
and the intervals $[j,(j+1))+[-2^{-j}, -2^{-(j+1)})  $ are disjoint for $ j $ even 
or $ j $ odd separately.
This concludes the the desired bound for the second term in \eqref{exolac}
and hence the proof of Corollary \ref{expcor}.
\end{proof}

\section{Multi-lacunary paraproducts}

In this section, we prove Theorem \ref{mlt}.
We begin by noting the natural decomposition of multi-lacunary sets.
\begin{lemma}

If $X$ is $(d,b)$-lacunary, then there exists a partition 
\[X=O_0\cup \dots \cup O_d\]
such that for every $i<d$ the set
\[X_i:=O_0\cup \dots \cup O_i\]
is $(i,b)$-lacunary and any two points 
 $\xi,\xi'$ in $O_{i+1}$ satisfy 
\[ \dist (\xi,\xi') \ge 2^{-b} \dist(\xi, X_i).\]
\end{lemma}

\begin{proof}
We successively decompose the limit sets 
in the definition of $(d,b)$-lacunarity.
\end{proof}

\subsection*{Reductions}
Consider the assumptions of Theorem \ref{mlt}.
We first break up the multiplier by decomposing the 
interval $[2^{-j},2^{2-j})$ into four equally long intervals
and intersecting these intervals with $[\xi_j,\zeta_j)$.
In other words, for integers $ 0 \le \beta \le 3 $, 
we choose numbers $\eta^{(\beta)}_j$ and $\zeta^{(\beta)}_j$ such that
\[1_{[\eta_j,\zeta_j)}=\sum_{m=0}^3 1_{[\eta_j^{(\beta)},\zeta_j^{(\beta)})}\]
and 
\begin{equation} \label{someeta}
2^{-j}+ \beta 2^{-j}\le \eta^{(\beta)}_j\le \zeta^{(\beta)}_j\le  2^{-j}+(\beta+1)2^{-j}
\end{equation}
Then the multiplier \eqref{expara} breaks up as a sum of four
analogous expressions and it suffices to show the desired bound for each of the summands.
We fix $\beta$ and suppress the dependency from the superscript for the rest of the proof. 

We also modify the sequence $ \xi_j $.
For each $j\in \N$, 
let $\xi_j'$ be the largest integer multiple of $2^{4-j}$ 
smaller than or equal to $\xi_j$.
Thanks to \eqref{xispacing}, the sequence $\xi'$ satisfies 
\begin{equation*}
0\le \xi'_j+2^{5-j}\le \xi'_{j-1}
\end{equation*}
Moreover, distance between $\xi_j'$ and $\xi_{j'}'$ for two indices
$j$ and $j'$ is comparable to the distance of $\xi_j$ and $\xi_{j'}$ with upper and lower factor at most
$2$. Hence the sequence $\xi_j'$ is $(d,b-2)$-lacunary.
Moreover, the intervals $[\xi_j',\xi_j)$ are pairwise disjoint.

\begin{lemma}
The multiplier $B_{m'}$ with
\[m'(\xi,\eta)=\sum_n 1_{[\xi_j',\xi_j)}(\xi) 1_{[\eta_j, \zeta_j)}(\eta)\]
satisfies  \eqref{holder} with constant depending only on $p_1,p_2$.
\end{lemma}
\begin{proof}
Recall that we denote by $M_I$ the linear Fourier multiplier for the interval $I$. We write 
 \begin{multline*}\|B_{m'}(f,g)\|_{p_3'}=\|\sum_{j} (M_{[\xi_j',\xi_j)}f)(M_{[\eta_j,\zeta_j)}g)\|_{p_3'}\\
 \le \|(\sum_{j} |M_{[\xi_j',\xi_j)}f|^2)^{1/2}\|_{p_1}
\|(\sum_{j}|M_{[\eta_j,\zeta_j)}g|)^{1/2}\|_{p_2} \\
\le C_{p_1,p_2}\|f\|_{p_1} \|g\|_{p_2},
\end{multline*}
where we applied Cauchy-Schwarz and H\"older to pass to the second line and
Theorem \ref{rdf} to bound the individual factors.
This proves the lemma.
\end{proof}

Because of the lemma, 
it suffices to estimate the multiplier \eqref{expara}
with $\xi_j$ replaced by $\xi_j'$. 
We shall do so and omit the prime in the forthcoming notation.
Let $X$ be the multi-lacunary image of the sequence $\xi_j$.
Let $Y$ be the collection of dyadic intervals $I$ 
such that $3I$ contains a point of $X$.
Let $Z$ be the collection of maximal dyadic intervals $I$ 
such that $3I$ does not contain any point of $X$. 
By maximality, the intervals of $Z$ are pairwise disjoint.
Let $Y_j$ be the collection of all dyadic intervals in $Y$ that have length $2^{4-j}$
and are contained in $[0,\xi_j)$. 
Let $Z_j$ be the collection of intervals in $Z$ which are contained in  
$[0,\xi_j)$ 
but not in any interval of $Y_j$.

\begin{lemma}\label{partlemma}
The intervals in $Y_j\cup Z_j$ have length at least $2^{4-j}$ and partition  $[0,\xi_j)$.
\end {lemma}

\begin{proof}
The intervals in $Y_j$ have length $2^{4-j}$ by definition. 
Let $I$ be an interval in $Z_j$ of length at most  $2^{4-j}$.
Let $J$ be the dyadic interval of length $2^{4-j}$ containing $I$. 
Then $J$ is contained in $[0,\xi_j)$,
because $\xi_j$ is an integer multiple of $2^{4-j}$
and $ J $ contains $ I \subset [0,\xi_j) $.
The interval $J$ is not in $Y_j$, 
since $I$ is not contained in any interval of $Y_j$. 
Hence $J$ must be contained in an interval $J' \in Z$. 
As the intervals in $Z$ are pairwise disjoint, 
$I$ cannot be strictly contained in $J'$, 
and hence $I$ has to be equal to $J'$.
Consequently, $ |I| = 2^{4-j}$. 
This proves the statement about the length of the intervals.

We turn to the claim about partitioning.
The intervals in $Y_j$ are pairwise disjoint, 
because they are dyadic intervals of equal length.
The intervals in $Z_j$ are pairwise disjoint, 
because they are maximal.  
By construction, no interval of $Z_j$ can be contained in any interval of $Y_j$.
Conversely, no interval in $Y_j$ can be contained in any interval of $Z_j$.
Indeed, $ 3I $ with $ I \in Z_j$ contains no point of $X$,
but $ 3J $ for $ J \in Y_j $ does contain a point of $ X $.
Hence $ J \nsubseteq I $.
Hence the intervals in $Y_j\cup Z_j$ are all pairwise disjoint. 

To prove that the intervals form a cover,
let $\xi$ be any point in $[0,\xi_j)$. 
Let $I$ be the dyadic interval of length $2^{4-j}$ containing $\xi$. 
It is contained in $[0,\xi_j)$. 
If $I$ is in $Y_j$, then $\xi$ is covered by intervals in $Y_j\cup Z_j$. 
If it is not in $Y_j$, then it is contained in an interval $J$ in $Z$.
The interval $J$ does not contain $\xi_j$ by definition of $ Z $.
Hence it is contained in $[0,\xi_j)$.
It is not contained in any interval of $Y_j$, 
and hence it is in $Z_j$. 
Again, $\xi$ is covered by $Y_j\cup Z_j$. 
This proves the partition statement.
\end{proof}

Using Lemma \ref{partlemma}, 
we may split $B_m(f,g)$ as
\begin{equation}\label{YZterms}
\sum_{j} \sum_{I\in Z_j}(M_{I}f)(M_{(\eta_j,\zeta_j)}g)
+\sum_{j} \sum_{I\in Y_j}(M_{I}f)(M_{(\eta_j,\zeta_j)}g).
\end{equation}
It suffices to estimate the terms separately.
Recall that we assume each $ \xi_j $ to be an integer multiple of $ 2^{4-j} $
and each $ (\eta_j,\zeta_j) $ to actually be $ (\eta_j^{(\beta)},\zeta_j^{(\beta)}) $
as defined in \eqref{someeta}.

\subsection*{First term}
Regrouping the sum in the first term and pairing with a dualizing function, 
the quantity to be estimated becomes 
\begin{multline}\label{zdual}
\int \sum_{I\in Z} \sum_{j: I\in Z_j}(M_{I}f)(M_{(\eta_j,\zeta_j)}g)(x) h(x)\, dx \\
= \int \sum_{I\in Z} (M_{I}f)
\left[\sum_{j: I\in Z_j}(M_{(\eta_j,\zeta_j)}g)(x)\right] M_{-I-[0,2^{-2}|I|]}h(x)\, dx.
\end{multline}
Here we used that the Fourier support of the product of $M_{I}f$ and  $M_{(\eta_j,\zeta_j)}g$ is contained in
\[I+(\eta_j,\zeta_j)\subset I+[0, 2^{2-j})\subset I+[0,2^{-2}|I|),\]
the first inclusion following from \eqref{someeta} and the second one from Lemma \ref{partlemma}.
We may therefore apply the adjoint of the Fourier restriction to this interval
to the dualizing function $h$ without changing the value of the duality pairing.

The intervals $I+[0,2^{-2}|I|)$ have bounded overlap as $I$ runs through $Z$.
To see this, partition each such interval as disjoint union of $I$ and $J_I$, 
where $J_I$ is the dyadic interval to the right of $I$ and has length $2^{-2}|I|$.   
Then $3J_I$ is contained in $3I$ and thus does not contain any point of $X$. 
Hence $J_I$ is contained in an interval of $Z$ with which it shares the left endpoint.
As the intervals in $Z$ are pairwise disjoint and each of them can only contain one interval $J_I$ as above,
the intervals $J_I$ are pairwise disjoint.

Consider an interval $I$ of $Z$ and assume it has nonzero contribution to \eqref{zdual}. 
Then there are integers $ j_I < j^{I} $ such that 
$I$ is contained in $[0,\xi_j)$ precisely if $j < j^I$, 
and it is not contained in any interval of $Y_j$ precisely if $j>j_I$. 
Hence we can notice
 \begin{equation}\label{phicut}
 \sum_{j: I\in Z_j} M_{1_{(\eta_j,\zeta_j)}}g =
  \sum_{j_I<  j< j^I} M_{1_{(\eta_j,\zeta_j)}}g =\tilde{g}*(\phi_{j_I} - \phi_{j^I})
  \end{equation}
where 
\[\tilde{g}=\sum_{j\in \N} M_{1_{(\eta_j,\zeta_j)}}g, \quad \widehat{\phi}_{j}(x)= \widehat{\phi}(2^{j}x) \]
and $\phi$ is a Schwartz function 
whose Fourier transform is supported on $[\beta, \beta+3]$ 
and is constant one on $[\beta+1, \beta +2 ]$.
Here we have used \eqref{someeta}.

The expression in \eqref{phicut} is bounded pointwise by a constant times the
Hardy--Littlewood maximal function ${\mathcal M}\tilde{g}$ of $\tilde{g}$.
Hence we estimate \eqref{zdual} by
\begin{multline*}
 C \int (\sum_{I\in Z}  |M_{1_{I}}f(x)|^2)^{1/2}
({\mathcal M}\tilde{g})(x) (\sum _{I\in Z} |M_{-I-[0,2^{-2}|I|]}h(x)|^2)^{1/2}\, dx \\
\le \| (\sum_{I\in Z}  |M_{I}f|^2)^{1/2}\|_{p_1}
\| \mathcal{M} \tilde{g}\|_{p_2} \| (\sum _{I\in Z} |M_{-I-[0,2^{-2}|I|]}h|^2)^{1/2}\|_{p_3}.
\end{multline*}
With Rubio de Francia's square function inequality from
Theorem \ref{rdf}, using the bounded overlap of $I+[0,2^{-2}|I|]$,
we estimate the last display by
\[ C_{p_1,p_2}\|f\|_{p_1} \|\tilde{g}\|_{p_2}\|h\|_{p_3}\le C_{p_1,p_2}\|f\|_{p_1} \|{g}\|_{p_2}\|h\|_{p_3}\]
In the last inequality, we have also used Theorem \ref{hmt}. 
The assumption on the total variation is obviously satisfied 
as the multiplier only jumps a bounded number of times by one in each of the test intervals.
This completes the bound for the first term in  \eqref{YZterms}.

\subsection*{Second term}
We turn to the second term in \eqref{YZterms}.
We decompose $Y$  as the union
\[Y=\bigcup_i Y^{(i)}\]
where $Y^{(i)}$ contains the intervals $I$ of $Y$ such that $3I$ contains a point of $O^{(i)}$ but
no point of any $O^{(i')}$ with $i'<i$. 
As the parameter $i$ only ranges form $0$ to $d$, 
it suffices to consider the $Y^{(i)}$ separately and prove a bound for 
\begin{equation}\label{Yiterm}
\sum_{j} \sum_{I\in Y_j\cap Y^{(i)}}(M_{I}f)(M_{(\eta_j,\zeta_j)}g).
\end{equation}

Let $W^{(i)}$ be the maximal intervals in $Y^{(i)}$.
In the next lemma, 
we single out two facts that we will need later.
The proof is similar, but not identical, to the argument that proved an analogous statement for $ Z $.

\begin{lemma}
  \label{lemma:whitneyprop}
  Let $ I \in W^{(i)} $. 
  If there is $ J \in W^{(i)} $ with 
  \[
  J \cap (I+[0,2^{-2}|I|)) \ne \varnothing,
  \]
  then $ |J| \ge 2^{-2}|I|$.
  In particular,
  the intervals $ I+[0,2^{-2}|I|)$ with $I \in W^{(i)}$
  have bounded overlap.
\end{lemma}
\begin{proof}
Take $I\in W^{(i)}$ and write
$
I+[0,2^{-2}|I|] = I\cup J_I
$
as a disjoint union.
Assume $J_I$ intersects another interval $J\in W^{(i)}$. 
It suffices to show that $J_I$ is contained in $J$.
Suppose $J$ is contained in $J_I$. 
Then $3J_I$ contains $3J$ and thus a point from $O^{(i)}$. 
On the other hand, 
$3J_I$ is contained in $3I$ 
and thus does not contain any point from any $O^{(i')}$ with $i'<i$. 
Hence $J_I$ is contained in an interval
of $W^{(i)}$. As the intervals of $W^{(i)}$ are pairwise disjoint, this interval must
be $J$ and hence $J$ is equal to $J_I$.
In particular, 
a right neighbor of $ I \in W^{(i)} $ has at least one quarter of the length of $ I $.
\end{proof}

To estimate \eqref{Yiterm}, 
we sort the intervals by their containment in maximal intervals, 
pair with a dualizing function, 
and realize the restriction of the Fourier support of the dualizing function with a multiplier as
\begin{equation}\label{bminus4}
\int \sum_{j} \sum_{J\in W^{(i)}} \sum_{ \substack{I\subset J \\ I\in Y_j\cap Y^{(i)} }}M_{I}f(x)M_{[\eta_j,\zeta_j)}g(x) M_{-J-[0,2^{-2}|J|)}(h)(x)dx.
\end{equation}
We break the innermost sum up by considering separately the cases $|I|> 2^{-b-4}|J|$ 
and $|I|\le  2^{-b-4}|J|$.

The sum with $|I|> 2^{-b-4}|J|$ is estimated by
\begin{multline*}
  \|(\sum_j \sum_{J\in W^{(i)}}(\sum_{\substack{I\subset J \\  I\in Y_j\cap Y^{(i)} \\  |I|> 2^{-b-4}|J|}} M_{I}f)^2)^{1/2}\|_{p_1} \\ 
  \times  \|(\sum_j |M_{(\eta_j,\zeta_j)}g|^2)^{1/2}\|_{p_2}\|(\sum_{J\in W^{(i)}} |M_{J+2^{-1}|J|}h|^2)^{1/2}\|_{p_3} .
\end{multline*}
The three factors are estimated by Theorem \ref{rdf}. 
It is clear that the intervals $(\eta_j,\zeta_j)$ are pairwise disjoint.
By Lemma \ref{lemma:whitneyprop}, $J+2^{-1}|J|$ have bounded overlap. 
For the first factor,
we use the disjointness of the various intervals in $Y_j$ to write it as
\[\|( \sum_{J\in W^{(i)}} \sum_{k=0}^{b+3} \sum_{\substack{I\subset J \\ I\in Y^{(i)} \\   |I|= 2^{-k}|J|}} |M_{1_{I}}f|^2)^{1/2}\|_{p_1} \]
As the various intervals $J$ are disjoint,
we obtain the estimate
\[C_{p_1}(b+4)\|f\|_{p_1}\]
by Theorem \ref{rdf}.
This completes the bound of the sum over $|I|> 2^{-b-4}|J|$ in \eqref{bminus4}.

We turn to the sum in \eqref{bminus4} over $|I|\le  2^{-b-4}|J|$.
Let $V^{(i)}$ be the collection of dyadic intervals $I$ 
such that there exists $J\in W^{(i)}$ with $I\subset J$ and $2^{b+4}|I|=|J|$
and at least one $ I' \subset I $ with $ I' \in Y_j\cap Y^{(i)}$.
Given a pair of neighboring intervals in $W^{(i)}$,
the right interval is at least one quarter as wide as left interval by Lemma \ref{lemma:whitneyprop}. 
The same property is inherited by the refined family $V^{(i)}$.
Accordingly, it suffices to estimate 
\begin{equation}\label{nestedmult}
\int \sum_{j} \sum_{J\in V^{(i)}} \sum_{\substack{I \subset J \\ I \in Y_j\cap Y^{(i)}} }M_{I}f(x)M_{[\eta_j,\zeta_j)}g(x) M_{-J-[0,2^{-2}|J|)}h(x)dx,
\end{equation}
where the Fourier support of $ h $ is now given in terms of an interval in $ V^{(i)} $.

For each $J\in V^{(i)}$, 
there is exactly one point of $O^{(i)}$ contained in $7J$.
Existence of at least one such a point follows by our requirement that $ J $ contains an interval from $ Y_j \cap Y^{(i)} $.
To show that there cannot be more than one,
suppose there were two such points $\xi,\xi'$. 
By the multi-lacunarity, 
there would then exist a point $\xi''$ of $O^{(i')}$ with $i'<i$ 
and
\[
|\xi''  - \xi| 
\le 2^b|\xi-\xi'|
\le 7 \cdot 2^{b} |J| \le \frac{7}{16} |I'| \] 
where $ I' $ is the interval in $ W^{(i)} $ with $ I' \supset J $.
This would imply $\xi'' \in 3I' $,
which contradicts the fact $ I' \in Y^{(i)}$.
Hence there is at most one point $ \xi_J \in 3J \cap O^{(i)} $,
and this point is the unique point from $ O^{(i)} $
withing the whole $ 7J $.

Let $\phi$ be a Schwartz function whose Fourier transform is 
supported on $ [-2^{-8}, 2^{-8} ]$ and
equal to one in $ [-2^{-9}, 2^{-9} ]$.
Write $\widehat{\phi}_{\xi,j}(\eta)=\widehat{\phi}(2^{j}(\eta-\xi))$.
We compare \eqref{nestedmult} with
\begin{equation}\label{nestedmodel}
\int \sum_{j} \sum_{J\in V^{(i)}} (\phi_{\xi_J,j}*f)(x)M_{(\eta_j,\zeta_j)}g(x) M_{-J-[0,2^{-2}|J|)}(h)(x)dx,
\end{equation}
where we have replaced $M_I$ by a convolution with $\phi_{\xi_J,j}$,
and the sum over $ I $ has been removed.
We first prove bounds on \eqref{nestedmodel},
and then we show how to pass from the original expression to \eqref{nestedmodel}.

We note that we may now further restrict the Fourier transform of $h$ to the interval
\[I_{J,j}=-[\xi_J-2^{-8-j}, \xi_J +2^{-8-j}]- [\eta_j,\zeta_j) . \]
Because $ \xi_J $ is a point unique to $ 7J $, 
because the intervals $ J $ are disjoint,
and because the sequences $ [\eta_j,\zeta_j) $ are subject to the condition \eqref{etaspacing},
we see that the intervals $ I_{J,j} $ form a disjoint family as $ J $ varies 
and disjoint and lacunary family as $ j $ varies.
We thus estimate \eqref{nestedmodel} by
\begin{multline*}
  \| (\sum_{J\in V^{(i)}}  \sup_{j} |\phi_{\xi_J,j}*(M_Jf)|^2)^{1/2}\|_{p_1} 
\|(\sum_j |M_{(\eta_j,\zeta_j)}g|^2)^{1/2}\|_{p_2} \\
\times
\|(\sum_{J\in V^{(i)}}\sum_j | M_{I_{J,j}}h|^2)^{1/2}\|_{p_3}
\end{multline*}
The last two factors we estimate with Theorem \ref{rdf}. 
In the first factor, 
we bound the convolution product by the Hardy--Littlewood maximal function 
and apply the Fefferman--Stein maximal inequality as well as Theorem \ref{rdf}.
This concludes the proof of the bound for \eqref{nestedmodel}.

It remains to bound the difference of \eqref{nestedmult} and \eqref{nestedmodel}.
Define
\[w_{j,J}  =  \left(\sum_{\substack{I \subset J \\ I \in Y_j\cap Y^{(i)}} } 1_I \right) - \widehat{\phi_{\xi_J,j}} 1_J
 .\]
Let $ \tilde{\psi}_{j,J} $ be a Schwartz function equal to one in $ [\xi_J - 2^{5-j}, \xi_J + 2^{5-j}] $ and zero outside $ [\xi_J - 2^{6-j}, \xi_J + 2^{6-j}] $.
Defining $ \psi_{j,J} = \tilde{\psi}_{j,J}  -  \widehat{\phi_{\xi_J,j}} $,
we then see that $ \psi_{j,J} $ is a Schwartz function
supported in $ [\xi_J - 2^{6-j}, \xi_J + 2^{6-j} ] $, 
vanishing in $ [-2^{-9-j}, 2^{-9-j} ]$
and satisfying 
\begin{equation}
\label{eq:psi_nested}
  w_{j,J}  = 1_J \psi_{j,J} \left(\sum_{\substack{I \subset J \\ I \in Y_j\cap Y^{(i)}} } 1_{I \setminus (\xi_J -2^{-9-j},\xi_J + 2^{-9-j} ) } \right) .
\end{equation}
Indeed, $ \xi_{J} $ is the unique point of $ \bigcup_{i' \le i} O^{(i)} $ in $ 7J $.
If there is $ I \in Y_j\cap Y^{(i)}$ with $ I \subset J $ and $ \partial I \setminus \partial J \ni \xi_J $, 
then $ I' $ with $ |I'| = |I| $ and $ I' \cap I = \{\xi_J\} $ is also in $ Y_j\cap Y^{(i)} $
and contained in $ J $.
Hence
\[
   \xi_J \notin \partial \left(\bigcup_{\substack{I \subset J \\ I \in Y_j\cap Y^{(i)}} } I\right) \setminus \partial J.
\]
On the other hand, the union above is contained in $  [\xi_J - 2^{5-j}, \xi_J + 2^{5-j}] $,
and so the equation \eqref{eq:psi_nested} is justified.

As we know the bounds for \eqref{nestedmodel},
it suffices to bound
\begin{multline*}
  \| (\sum_{J\in V^{(i)}}  \sum_{j}  | \sum_{\substack{I \subset J \\ I \in Y_j\cap Y^{(i)}} } M_I M_Jf- \phi_{\xi_J,j}*(M_Jf)|^2)^{1/2}\|_{p_1} \\ 
\times
\|(\sum_j |M_{(\eta_j,\zeta_j)}g|^2)^{1/2}\|_{p_2} \|(\sum_{J\in V^{(i)}}  | M_{(-J-[0,2^{-2}|J|))}h|^2)^{1/2}\|_{p_3}.
\end{multline*}
Here the last two factors are readily estimated by Theorem \ref{rdf}.
We focus on the first factor.
By the observation \eqref{eq:psi_nested},
we can bound the multiplier operator associated with frequency symbol $ \psi_{j,J} $
by the Hardy--Littlewood maximal function and apply the Fefferman--Stein inequality to control the first factor by 
\[
  \| (\sum_{J\in V^{(i)}}  \sum_{j}  |\sum_{\substack{I \subset J \\ I \in Y_j\cap Y^{(i)}} } M_{ I \setminus ( \xi_J - 2^{-9-j}, \xi_J + 2^{-9-j}  ) }  M_Jf|^2)^{1/2}\|_{p_1} .
\]
However, the intervals above have bounded overlap,
and hence we are in a position to apply Theorem \ref{rdf} for one last time.
This completes the proof of Theorem \ref{mlt}.
\bibliography{ref}

\begin{thebibliography}{10}

\bibitem{MR2494456}
M.~Bateman.
\newblock Kakeya sets and directional maximal operators in the plane.
\newblock {\em Duke Math. J.}, 147(1):55--77, 2009.

\bibitem{MR812567}
J.~Bourgain.
\newblock Estimations de certaines fonctions maximales.
\newblock {\em C. R. Acad. Sci. Paris S\'{e}r. I Math.}, 301(10):499--502,
  1985.

\bibitem{MR949003}
A.~Carbery.
\newblock Differentiation in lacunary directions and an extension of the
  {M}arcinkiewicz multiplier theorem.
\newblock {\em Ann. Inst. Fourier (Grenoble)}, 38(1):157--168, 1988.

\bibitem{MR934617}
R.~Coifman, J.~L. Rubio~de Francia, and S.~Semmes.
\newblock Multiplicateurs de {F}ourier de {$L^p({\bf R})$} et estimations
  quadratiques.
\newblock {\em C. R. Acad. Sci. Paris S\'{e}r. I Math.}, 306(8):351--354, 1988.

\bibitem{MR3000982}
C.~Demeter and S.~Z. Gautam.
\newblock Bilinear {F}ourier restriction theorems.
\newblock {\em J. Fourier Anal. Appl.}, 18(6):1265--1290, 2012.

\bibitem{MR3846321}
F.~Di~Plinio and I.~Parissis.
\newblock A sharp estimate for the {H}ilbert transform along finite order
  lacunary sets of directions.
\newblock {\em Israel J. Math.}, 227(1):189--214, 2018.

\bibitem{MR4126303}
F.~Di~Plinio and I.~Parissis.
\newblock On the maximal directional {H}ilbert transform in three dimensions.
\newblock {\em Int. Math. Res. Not. IMRN}, (14):4324--4356, 2020.

\bibitem{MR3453362}
F.~Di~Plinio and C.~Thiele.
\newblock Endpoint bounds for the bilinear {H}ilbert transform.
\newblock {\em Trans. Amer. Math. Soc.}, 368(6):3931--3972, 2016.

\bibitem{MR284802}
C.~Fefferman and E.~M. Stein.
\newblock Some maximal inequalities.
\newblock {\em Amer. J. Math.}, 93:107--115, 1971.

\bibitem{MR2113017}
L.~Grafakos and X.~Li.
\newblock Uniform bounds for the bilinear {H}ilbert transforms. {I}.
\newblock {\em Ann. of Math. (2)}, 159(3):889--933, 2004.

\bibitem{MR2197068}
L.~Grafakos and X.~Li.
\newblock The disc as a bilinear multiplier.
\newblock {\em Amer. J. Math.}, 128(1):91--119, 2006.

\bibitem{MR2434308}
R.~L. Jones, A.~Seeger, and J.~Wright.
\newblock Strong variational and jump inequalities in harmonic analysis.
\newblock {\em Trans. Amer. Math. Soc.}, 360(12):6711--6742, 2008.

\bibitem{MR1425870}
M.~Lacey and C.~Thiele.
\newblock {$L^p$} estimates for the bilinear {H}ilbert transform.
\newblock {\em Proc. Nat. Acad. Sci. U.S.A.}, 94(1):33--35, 1997.

\bibitem{MR1491450}
M.~Lacey and C.~Thiele.
\newblock {$L^p$} estimates on the bilinear {H}ilbert transform for
  {$2<p<\infty$}.
\newblock {\em Ann. of Math. (2)}, 146(3):693--724, 1997.

\bibitem{MR1689336}
M.~Lacey and C.~Thiele.
\newblock On {C}alder\'{o}n's conjecture.
\newblock {\em Ann. of Math. (2)}, 149(2):475--496, 1999.

\bibitem{MR1619285}
M.~T. Lacey and C.~M. Thiele.
\newblock On {C}alder\'{o}n's conjecture for the bilinear {H}ilbert transform.
\newblock {\em Proc. Natl. Acad. Sci. USA}, 95(9):4828--4830, 1998.

\bibitem{MR420837}
D.~L\'{e}pingle.
\newblock La variation d'ordre {$p$} des semi-martingales.
\newblock {\em Z. Wahrscheinlichkeitstheorie und Verw. Gebiete},
  36(4):295--316, 1976.

\bibitem{MR2413217}
X.~Li.
\newblock Uniform estimates for some paraproducts.
\newblock {\em New York J. Math.}, 14:145--192, 2008.

\bibitem{MR3337797}
V.~Lie.
\newblock On the boundedness of the bilinear {H}ilbert transform along
  ``non-flat'' smooth curves.
\newblock {\em Amer. J. Math.}, 137(2):313--363, 2015.

\bibitem{MR3763348}
V.~Lie.
\newblock On the boundedness of the bilinear {H}ilbert transform along
  ``non-flat'' smooth curves. {T}he {B}anach triangle case {$(L^r,\ 1\leq
  r<\infty)$}.
\newblock {\em Rev. Mat. Iberoam.}, 34(1):331--353, 2018.

\bibitem{MR1456993}
Y.~Meyer and R.~Coifman.
\newblock {\em Wavelets}, volume~48 of {\em Cambridge Studies in Advanced
  Mathematics}.
\newblock Cambridge University Press, Cambridge, 1997.
\newblock Calder\'{o}n-Zygmund and multilinear operators, Translated from the
  1990 and 1991 French originals by David Salinger.

\bibitem{MR2701349}
C.~Muscalu.
\newblock {\em L(p) estimates for multilinear operators given by singular
  symbols}.
\newblock ProQuest LLC, Ann Arbor, MI, 2000.
\newblock Thesis (Ph.D.)--Brown University.

\bibitem{MR3052498}
C.~Muscalu and W.~Schlag.
\newblock {\em Classical and multilinear harmonic analysis. {V}ol. {I}}, volume
  137 of {\em Cambridge Studies in Advanced Mathematics}.
\newblock Cambridge University Press, Cambridge, 2013.

\bibitem{MR1979774}
C.~Muscalu, T.~Tao, and C.~Thiele.
\newblock Uniform estimates on multi-linear operators with modulation symmetry.
\newblock {\em J. Anal. Math.}, 88:255--309, 2002.

\bibitem{MR1945289}
C.~Muscalu, T.~Tao, and C.~Thiele.
\newblock Uniform estimates on paraproducts.
\newblock {\em J. Anal. Math.}, 87:369--384, 2002.

\bibitem{MR3432267}
J.~Parcet and K.~M. Rogers.
\newblock Directional maximal operators and lacunarity in higher dimensions.
\newblock {\em Amer. J. Math.}, 137(6):1535--1557, 2015.

\bibitem{MR850681}
J.~L. Rubio~de Francia.
\newblock A {L}ittlewood-{P}aley inequality for arbitrary intervals.
\newblock {\em Rev. Mat. Iberoamericana}, 1(2):1--14, 1985.

\bibitem{MR613033}
P.~Sj\"{o}gren and P.~Sj\"{o}lin.
\newblock Littlewood-{P}aley decompositions and {F}ourier multipliers with
  singularities on certain sets.
\newblock {\em Ann. Inst. Fourier (Grenoble)}, 31(1):vii, 157--175, 1981.

\bibitem{MR1933076}
C.~Thiele.
\newblock A uniform estimate.
\newblock {\em Ann. of Math. (2)}, 156(2):519--563, 2002.

\end{thebibliography}

\bibliographystyle{abbrv}

\end{document}